\documentclass{amsart}

\usepackage{amsmath}
\usepackage{url}
\usepackage{amsthm}
\usepackage{amsfonts}
\usepackage{hyperref}
\usepackage{mathrsfs}
\usepackage{color}
\usepackage{enumerate}
\usepackage{cite}
\usepackage{amssymb}

\newtheorem{theorem}{Theorem}[section]

\newtheorem{lemma}[theorem]{Lemma}
\newtheorem{proposition}[theorem]{Proposition}

\theoremstyle{definition}
\newtheorem{definition}[theorem]{Definition}

\theoremstyle{remark}

\numberwithin{equation}{section}

\newcommand{\R}{\mathbb{R}}

\newcommand{\abs}[1]{\left|#1\right|}
\newcommand{\cone}{{\mathcal{C}_{\alpha}}}
\newcommand{\dcone}{{\mathcal{D}_{\alpha}}}

\begin{document}
\title[]{Rotational Symmetry of Asymptotically Conical Mean Curvature Flow Self-expanders}
\author[Frederick~Fong]{Frederick~Tsz-Ho~Fong}
\author[Peter~McGrath]{Peter~McGrath}

\date{}
\address{Department of Mathematics, Hong Kong University of Science and Technology, Clear Water Bay, Kowloon, Hong Kong}\email{frederick.fong@ust.hk}
\address{Department of Mathematics, Brown University, Providence,
RI 02912}  \email{peter\_mcgrath@math.brown.edu}

\maketitle
\begin{abstract}
In this article, we examine complete, mean-convex self-expanders for the mean curvature flow whose ends have decaying principal curvatures. We prove a Liouville-type theorem associated to this class of self-expanders. As an application, we show that mean-convex self-expanders which are asymptotic to $O(n)$-invariant cones are rotationally symmetric.
\end{abstract}

\section{Introduction}

Let $F : \Sigma^n \to \R^{n+1}$ be an orientable hypersurface with Gauss map $\nu$. We use the following convention for the second fundamental form and the mean curvature:
\begin{align*}
h_{ij} & = \left\langle\frac{\partial^2 F}{\partial u_i\partial u_j},\,\nu\right\rangle = -\left\langle\frac{\partial F}{\partial u_i},\,\frac{\partial \nu}{\partial u_j}\right\rangle\\
H & = g^{ij}h_{ij}.
\end{align*}
Under this convention, $\Sigma$ is said to be a \emph{self-expander} for the mean curvature flow (MCF) if the following holds:
\begin{equation}
\label{eq:MCF_Expander}
H = \frac{1}{2}\langle F, \nu\rangle.
\end{equation}
If $F$ satisfies \eqref{eq:MCF_Expander}, then the family of immersions $F_t := \sqrt{t} F$, which evolves by expanding $\Sigma$ homothetically, will satisfy the mean curvature flow (MCF):
\begin{equation}
\label{eq:MCF_perp}
\left(\frac{\partial F_t}{\partial t}\right)^\perp = H_t \nu_t,
\end{equation}
where $H_t$ and $\nu_t$ are respectively the mean curvature and Gauss map of $F_t$.

Motivation for the study of self-expanders goes back to work of Ecker-Huisken \cite{ecker1989mean}, who studied MCF evolutions of entire graphical immersions $F: \Sigma^{n} \to \R^{n+1}$.  Assuming a linear growth and Lipschitz condition on the graph function, they were able to show the flow exists for all time.  Furthermore, assuming the initial data satisfied a ``flatness'' condition
\begin{align}
\label{Eecker}
\left| \langle F, \nu \rangle \right| \leq C ( 1+ |F|)^{1- \delta}
\end{align}
(for some $\delta>0$ and $C>0$), they showed that appropriate parabolic blow-downs of the flow converge to a self-expander.    Stavrou \cite{stavrou1998selfsimilar} later proved this same result under weaker hypotheses, by relaxing the condition \eqref{Eecker} to the condition that the graphical function $w$ have a unique tangent cone at infinity.  Self-expanders can thus be viewed as models for the long term behavior of MCF.

Ecker-Huisken also observed that each $O(n)$-invariant cone $\cone$ with cone angle $\alpha$ possesses a unique \emph{graphical} self-expander asymptotic to $\cone$.  This self-expander is $O(n)$-invariant.  Later, Angenent, Chopp and Ilmanen \cite{angenent1995computed} considered self-expanders asymptotic to $O(n)$-invariant \emph{double} cones $\mathcal{D}_\alpha$ in $\R^{n+1}$.  Interestingly, by relaxing the graphicality assumption, they were able to find cones $\cone$ with more than one $O(n)$-invariant self-expanding evolution.  In acoordance with Ecker-Huisken's work above, for each $\alpha \in (0, \pi/2)$ they were able to find a two-sheeted evolution described as a bigraph over the base plane.  However, they found that for each $n\geq 3$, there is a critical cone angle $\alpha_{\textup{crit}}(n)$ such that a one-sheeted evolution of $\dcone$ exists whenever $\alpha \geq \alpha_{\textup{crit}}(n)$. 
 
 Subsequently, Helmensdorfer \cite{helmensdorfer2012solitons} rigorously proved the existence of a \emph{second} one-sheeted rotationally symmetric expander asymptotic to $\dcone$ for each $\alpha > \alpha_{\textup{crit}}(n)$.  In short: for large enough $\alpha$, the double cone $\dcone$ has at least three distinct rotationally symmetric self-expanding evolutions.

There are analogues and generalizations of the above results in higher codimension, especially for the Lagrangian mean curvature flow (LMCF).  Neves and Tian \cite{Neves:2013hh} showed that blow-downs of eternal solutions of LMCF converge to LMCF self-expanders.  Chau, Chen and He \cite{Chau:2009wu, Chau:2011ff} studied LMCF on entire Lagrangian graphs and proved uniqueness of graphical LMCF self-expanders asymptotic to Lagrangian cones over a real plane.  Examples of LMCF self-expanders asymptotic to Schoen-Wolfson's cones were constructed by Lee and Wang \cite{Lee:2009ty, Lee:2010bu}, and examples asymptotic to a pair of intersecting Lagrangian planes were constructed by Joyce, Lee and Tsui in \cite{Joyce:2010uw}.   Lotay and Neves \cite{Lotay:2013jf} later established a uniqueness theorem for LMCF self-expanders in the latter setting. 

Returning our attention back to self-expanding hypersurfaces, the multiplicity of rotationally symmetric examples constructed in Ecker-Huisken \cite{ecker1989mean}, Angenent-Chopp-Ilmanen \cite{angenent1995computed} and Helmensdorfer \cite{helmensdorfer2012solitons} suggest an intriguing and intricate relationship between MCF self-expanders, rotational symmetry and asymptotic cones. In \cite[P.1940]{angenent1995computed}, the authors posted a question of whether there are non-rotationally symmetric MCF solutions coming out of a double cone.

Partly motivated by this question and by the rotationally invariant known examples, we examine complete mean-convex MCF self-expanders whose ends have decaying principal curvatures (including those with conical ends and planar ends). We prove a Liouville-type theorem (Theorem \ref{Tliouville}) for eigenfunctions of a stability operator $L$ for a weighted area functional (analogous to the one studied by Colding and Minicozzi in \cite{colding2012generic} for self-shrinkers). Using this, we show that complete, mean-convex MCF self-expanders asymptotic to $O(n)$-invariant cones are themselves $O(n)$-invariant. Precisely, we prove the following:

\begin{theorem}
\label{TMain}
Suppose $\Sigma$ is a complete, mean-convex self-expander. If either:
\begin{itemize}
	\item $\Sigma$ has only one end $E$ which is $C^2$-asymptotic (in the sense of Definition \ref{Ddecay}) to a round cone $\cone$; or
	\item $\Sigma$ has only two ends $E_\alpha$ and $E_\beta$ which are $C^2$-asymptotic (in the sense of Definition \ref{Ddecay}) to round cones with the same rotation axis,
\end{itemize}
then $\Sigma$ is rotationally symmetric about the axis of the cone.

Furthermore, if in addition $\Sigma$ is convex, then the same conclusion holds by assuming the ends are only $C^1$-asymptotically conical.
\end{theorem}

Our approach of proving rotational symmetry is inspired by works by Schoen \cite{Schoen:1983} and Solomon-Simon \cite{Simon1986} on minimal surfaces, works by Brendle \cite{brendle2013rotational,brendle2014rotational}, Chodosh \cite{chodosh2014expanding}, Chodosh-Fong \cite{chodosh2016rotational} on Ricci solitons, works by Haslhofer \cite{haslhofer2015uniqueness} and Bourni-Langford \cite{Bourni:2016tn} on translating MCF solutions, as well as by Drugan-Fong-Lee's recent work \cite{drugan2016rotational} on self-expanders of the inverse mean curvature flow. 

It is interesting to compare Theorem \ref{TMain} to analogous uniqueness results for MCF self-shrinkers.  Wang \cite{wang2014uniqueness} proved there is at most one embedded, self-shrinking end asymptotic to each (not necessarily $O(n)$-invariant) regular cone $\mathcal{C}$ with vertex at $0$.  This is in contrast to the situation for self-expanders, where multiple examples (e.g. \cite{helmensdorfer2012solitons, angenent1995computed} as described above) asymptotic to the same cone exist. Wang \cite{wang2016uniqueness} also proved a similar uniqueness theorem for self-shrinkers with asymptotically cylindrical ends, under infinite-order asymptotic assumptions.   Wang's methods rely on the unique solvability of a backwards heat type equation.  

\textbf{Acknowledgement.} The first named author would like to thank Hojoo Lee for many inspiring discussions about self-similar solutions in various curvature flows. He is partially supported by the Hong Kong RGC Early Career Grant (26301316).  The second named author thanks his advisor, Nikolaos Kapouleas, for his support and guidance.  

\section{Stability Operator for Self-Expanders}
In this section, we derive elliptic equations for the mean curvature and the support functions of rotational Killing vector fields. We use variational methods and employ the following formulae, which are well-known in the literature of geometric flows of hypersurfaces.

\begin{lemma}[c.f. \cite{BenAndrews:1994jx}, \cite{Huisken:1999wk}]
Suppose $F_s : \Sigma^n \to \R^{n+1}$ is a smooth family of hypersurfaces satisfying $\displaystyle{\frac{\partial F_s}{\partial s} = \varphi_s \nu_s}$, where each $\varphi_s : \Sigma \to \R$ is a smooth scalar function, and $\nu_s$ is the Gauss map of $F_s$. Then
\begin{align}
\label{eq:nu} \frac{\partial \nu_s}{\partial s} & = -\nabla_s \varphi_s\\
\label{eq:H} \frac{\partial H_s}{\partial s} & = \left(\Delta_s + \abs{A_s}^2\right)\varphi_s,
\end{align}
where $H_s$ and $A_s$ are the mean curvature and second fundamental form of $F_s$, and $\Delta_s$ and $\nabla_s$ are respectively the Laplacian and gradient with respect to $F_s$.
\end{lemma}
\begin{proof}
See e.g. \cite[Theorem 3.7]{BenAndrews:1994jx} or \cite[Theorem 3.2]{Huisken:1999wk}.
\end{proof}

The first identity below is an elliptic equation for the mean curvature. It can be proved by using local coordinates (e.g. Ding \cite[P.9]{Ding:2015vs}), or moving frames (see the self-shrinkers analog in Colding-Minicozzi \cite[P.780]{colding2012generic}). Nonetheless, we give a variational proof to unify with the proofs of Lemma \ref{lma:Laplace_f} and Lemma \ref{lma:Laplace_V}.

\begin{lemma}[c.f. \cite{Ding:2015vs}]
\label{lma:Laplace_H}
Suppose $F : \Sigma^n \to \R^{n+1}$ is a self-expander. Then
\begin{equation}
\label{eq:Laplace_H}
\Delta H + \frac{1}{2}\langle F, \nabla H\rangle + \left(\abs{A}^2 + \frac{1}{2}\right)H = 0,
\end{equation}
where $\Delta := \Delta_\Sigma$ and $\nabla := \nabla_\Sigma$.
\end{lemma}

\begin{proof}
Recall that given a self-expander, the family $F_t := \sqrt{t}F$ satisfies \eqref{eq:MCF_perp}. By pre-composing a suitable re-parametrization $\Phi_t$ (where $\Phi_1 = \textup{id}$), we can arrange that
\begin{equation}
\label{eq:MCF}
\frac{\partial}{\partial t} \left(F_t \circ \Phi_t\right) = H(F_t \circ \Phi_t) \,\nu(F_t \circ \Phi_t).
\end{equation}
By \eqref{eq:H}, we have
\begin{equation}
\label{eq:H1}
\left.\frac{\partial}{\partial t}\right|_{t=1}H(F_t \circ \Phi_t) = \left(\Delta + \abs{A}^2\right)H.
\end{equation}
On the other hand, by invariance under re-parametrization and the self-expander equation \eqref{eq:MCF_Expander}, we have
\[H(F_t \circ \Phi_t) = \frac{1}{\sqrt{t}} H(F \circ \Phi_t) = \frac{1}{2t}\langle F_t \circ \Phi_t, \nu(F_t \circ \Phi_t)\rangle.\]
By differentiating both sides, we find
\begin{equation}
\label{eq:H2}
\begin{aligned}
\left.\frac{\partial}{\partial t}\right|_{t=1} H(F_t \circ \Phi_t) & = \left.\frac{\partial}{\partial t}\right|_{t=1}\frac{1}{2t}\langle F_t \circ \Phi_t, \, \nu(F_t \circ \Phi_t)\rangle\\
 & = -\frac{1}{2}\langle F, \nu\rangle + \frac{1}{2}\left.\frac{\partial}{\partial t}\right|_{t=1}\langle F_t \circ \Phi_t, \, \nu(F_t \circ \Phi_t)\rangle\\
 & = -H + \frac{1}{2} \langle H\nu, \nu\rangle + \frac{1}{2}\left\langle F, -\nabla H\right\rangle\\
 & = -\frac{1}{2}H - \frac{1}{2}\langle F, \nabla H\rangle,
\end{aligned}
\end{equation}
where we have used \eqref{eq:nu}.  \eqref{eq:Laplace_H} then follows by combining \eqref{eq:H1} and \eqref{eq:H2}.
\end{proof}

We next consider rotational Killing vector fields. Fixing an axis $\ell$, say the $x^{n+1}$-axis in $\R^{n+1}$, observe that a rotational vector field $R$ fixing $\ell$ (such as $- x^2 \partial_1 + x^1 \partial_2$) has round circles as integral curves.  In the rest of this article, if $F : \Sigma \to \R^{n+1}$ is an immersion and $R$ is as above, we define $f_R = \langle R(F), \nu(F)\rangle$.  If $f_R$ vanishes identically, then $R$ is a tangent field for $\Sigma$ and $\Sigma$ is invariant along its integral curves.

To show $\Sigma$ is rotationally symmetric, we will show that $f_R = \langle R, \nu\rangle$ vanishes for \emph{any} rotational Killing vector field about the given axis (c.f. \cite{haslhofer2015uniqueness} and \cite{drugan2016rotational}). This inner product plays a similar role as the Lie derivative $\mathcal{L}_U g$ in \cite{brendle2013rotational, brendle2014rotational, chodosh2014expanding, chodosh2016rotational}.

\begin{lemma}
\label{lma:Laplace_f}
Suppose $F : \Sigma^n \to \R^{n+1}$ is a self-expander and $R$ is a rotational Killing vector field. Then
\begin{equation}
\label{eq:Laplace_f}	
\Delta f_R + \frac{1}{2}\langle F, \nabla f_R\rangle + \left(\abs{A}^2 - \frac{1}{2}\right)f_R = 0.
\end{equation}
\end{lemma}
\begin{proof}
Let $\Psi_s$ be the 1-parameter family of diffeomorphisms of $\R^{n+1}$ generated by the vector field $R$, i.e. satisfying $\displaystyle{\frac{\partial}{\partial s}\Psi_s = R \circ \Psi_s}$ and $\Psi_0 = \textup{id}$. Then, the composition $\Psi_s \circ F$ satisfies
\[\frac{\partial}{\partial s}\left(\Psi_s \circ F\right) = R(\Psi_s \circ F).\]
As in the proof of Lemma \ref{lma:Laplace_H}, there exists a suitable re-parametrization $\Phi_s$ with $\Phi_0 = \textup{id}$ such that $\Psi_s \circ F \circ \Phi_s$ satisfies
\begin{equation}
\frac{\partial}{\partial s}\left(\Psi_s \circ F \circ \Phi_s\right) = \underbrace{R(\Psi_s \circ F \circ \Phi_s)^{\perp}}_{\text{normal component}} = f_R(\Psi_s \circ F \circ \Phi_s)\,\nu(\Psi_s \circ F \circ \Phi_s).
\end{equation}
By \eqref{eq:H}, we obtain
\begin{equation}
\label{eq:H3}
\left.\frac{\partial}{\partial s}\right|_{s=0} H(\Psi_s \circ F \circ \Phi_s) = \left(\Delta + \abs{A}^2\right)f_R.
\end{equation}

On the other hand, $H$ is invariant under rotations, so from the self-expander equation \eqref{eq:MCF_Expander},
\[H(\Psi_s \circ F \circ \Phi_s) = \frac{1}{2}\langle \Psi_s \circ F \circ \Phi_s, \, \nu(\Psi_s \circ F \circ \Phi_s)\rangle.\]
By differentiating both sides, we get
\begin{equation}
\label{eq:H4}
\begin{aligned}
 \left.\frac{\partial}{\partial s}\right|_{s=0} H(\Psi_s \circ F \circ \Phi_s) & = \frac{1}{2}\langle f_R \nu, \,\nu\rangle + \frac{1}{2}\langle F, -\nabla f_R\rangle\\
 & = \frac{1}{2}f_R - \frac{1}{2}\langle F, \nabla f_R\rangle,
\end{aligned}
\end{equation}
where we have used \eqref{eq:nu}.  \eqref{eq:Laplace_f} then follows by combining \eqref{eq:H3} and \eqref{eq:H4}. 
\end{proof}

Next we state and sketch the proof of a similar identity for the inner product $\langle V, \nu\rangle$ where $V$ is a constant vector field in $\R^{n+1}$. We will not use this identity in the later part of this article, yet we expect that it may be useful when dealing with self-expanders with planar ends.
\begin{lemma}
\label{lma:Laplace_V}
Suppose $F : \Sigma^n \to \R^{n+1}$ is a self-expander and $V$ is a constant vector field in $\R^{n+1}$. Then,
\begin{equation}
\label{eq:Laplace_V}
\Delta \langle V, \nu \rangle + \frac{1}{2}\left\langle F, \nabla\langle V, \nu\rangle\right\rangle + \abs{A}^2\langle V, \nu \rangle = 0.
\end{equation}
\end{lemma}
\begin{proof}
We only sketch the proof, since it is similar to the proofs of Lemma \ref{lma:Laplace_H} and Lemma \ref{lma:Laplace_f}.  Consider the family of translating hypersurfaces $F_\tau := F + \tau V$ so that $\left(\frac{\partial F_{\tau}}{\partial\tau}\right)^\perp = \langle V, \nu\rangle \nu$. From the fact that
\[H(F_\tau) = H(F) = \frac{1}{2}\langle F, \nu(F)\rangle = \frac{1}{2}\langle F, \nu(F_\tau)\rangle,\]
one can differentiate both sides at $\tau = 0$ and show \eqref{eq:Laplace_V} using the variational formulae \eqref{eq:nu} and \eqref{eq:H}.
\end{proof}

In view of \eqref{eq:Laplace_H}, \eqref{eq:Laplace_f} and \eqref{eq:Laplace_V}, we define a stability operator $L$, analogous to the operator for mean curvature self-shrinkers introduced by Colding-Minicozzi in \cite{colding2012generic}, by 
\begin{equation}
\label{eq:L}
L\varphi := \Delta\varphi + \frac{1}{2}\langle F, \nabla\varphi\rangle +\left(\abs{A}^2- \frac{1}{2}\right) \varphi.
\end{equation}
Then, \eqref{eq:Laplace_H}, \eqref{eq:Laplace_f} and \eqref{eq:Laplace_V} can be expressed as
\begin{equation}
\label{Eeval}
\begin{aligned}
LH & = -H\\
Lf_R & = 0\\
L\langle V, \nu\rangle & = -\frac{1}{2} \langle V, \nu \rangle.
\end{aligned}
\end{equation}

\begin{lemma}
Let $\Sigma$ be a mean-convex self-expander.  Then, for any $\varepsilon > 0, \lambda \in \R$ and $f\in C^2(\Sigma)$ such that $\left(L+\frac{1}{2}\right)f = \lambda f$,
\begin{equation}
\label{eq:Quotient}
\begin{aligned}
 \Delta\left(\frac{f}{H+\varepsilon}\right)& = \left \langle - \frac{1}{2} F - \frac{2}{H+\varepsilon} \nabla (H+\varepsilon), \nabla\left(\frac{f}{H+\varepsilon}\right)\right\rangle\\
 &  + \frac{f}{(H+\varepsilon)^2} \left(\left(\lambda + \frac{1}{2}\right)H - \varepsilon \left( |A|^2 - \lambda\right)\right)	.
\end{aligned}
\end{equation}
\end{lemma}

\begin{proof}
From Lemma \eqref{eq:Laplace_H} and that $Lf + \frac{1}{2}f = \lambda f$, we have
\begin{equation}
\label{EHf}
\begin{aligned}
\Delta (H+\varepsilon)  & = - \frac{1}{2} \langle F, \nabla(H+\varepsilon)\rangle - \left( |A|^2 - \frac{1}{2}\right)(H+\varepsilon)\\
 & \quad - H + \varepsilon\left(|A|^2 - \frac{1}{2}\right),\\
 \Delta f & = - \frac{1}{2} \langle F , \nabla f\rangle - \left(|A|^2 - \lambda\right) f.
\end{aligned}
\end{equation}
For simplicity, we denote
\[u := \frac{f}{H+\varepsilon}.\]
Using \eqref{EHf} and the quotient identity
\begin{align*}
\Delta \left( \frac{g}{h} \right) = \frac{h\Delta g - g \Delta h}{h^2} - \frac{2}{h}\left \langle \nabla h, \nabla \left(\frac{g}{h}\right)\right\rangle,
\end{align*}
which holds for any $g, h \in C^2(\Sigma)$, we compute
\begin{align*}
& \Delta\left(\frac{f}{H+\varepsilon}\right)\\
& = (H+\varepsilon)^{-2}\left[(H+\varepsilon) \Delta f - f\Delta (H+\varepsilon)\right] - \frac{2}{H+\varepsilon} \left \langle  \nabla (H+\varepsilon), \nabla u\right \rangle \\
& = (H+\varepsilon)^{-2}\left[ (H+\varepsilon) \left( - \frac{1}{2}\langle F, \nabla f\rangle - \left(|A|^2 - \lambda\right) f\right)\right]\\ 
 & \quad + (H+\varepsilon)^{-2}f\left[   \frac{1}{2} \langle F, \nabla(H+\varepsilon)\rangle+  \left( |A|^2 - \frac{1}{2}\right)(H+\varepsilon)+ H - \varepsilon\left(|A|^2 - \frac{1}{2}\right)  \right]\\ 
 & \quad - \frac{2}{H+\varepsilon} \left \langle  \nabla (H+\varepsilon), \nabla u\right \rangle \\
&= - \frac{1}{2} \left \langle F, \frac{(H+\varepsilon) \nabla f - f \nabla(H+\varepsilon)}{(H+\varepsilon)^2} \right\rangle + \frac{f}{(H+\varepsilon)^2} \left(\left(\lambda+\frac{1}{2}\right)H - \varepsilon \left( |A|^2 - \lambda\right)\right)\\
& \quad - \frac{2}{H+\varepsilon} \left \langle  \nabla (H+\varepsilon), \nabla u\right \rangle \\
&= \left \langle - \frac{1}{2} F - \frac{2}{H+\varepsilon} \nabla (H+\varepsilon), \nabla u\right\rangle + \frac{f}{(H+\varepsilon)^2} \left(\left(\lambda+\frac{1}{2}\right)H - \varepsilon \left( |A|^2 - \lambda\right)\right).
 \end{align*}
This completes the proof of \eqref{eq:Quotient}.
\end{proof}

Next, we prove a Liouville-type theorem concerning the stability operator $L$ on the class of mean-convex self-expanders $\Sigma$ with decaying curvature in the sense that
	\begin{equation}
	\label{eq:Decay_A}
	\lim_{r \to \infty} \sup_{\Sigma \backslash B_r(0)} \abs{A} =0.
	\end{equation}
	Here $B_r(0)$ is the ball in $\R^{n+1}$ with radius $r$ centered at origin.
Examples of hypersurfaces satisfying \eqref{eq:Decay_A} include those which are asymptotically conical (to be discussed in the next section), and those which are asymptotically planar.

\begin{theorem}
\label{Tliouville}
Let $\Sigma$ be a mean-convex self-expander. Suppose $\Sigma$ satisfies \eqref{eq:Decay_A} and $f\in C^2(\Sigma)$ satisfies $\left(L+\frac{1}{2}\right)f = \lambda f$, where $\lambda > 0$ and
\begin{equation}
\label{eq:Decay_f}
\lim_{r\to \infty} \sup_{\Sigma \backslash B_r(0)} \abs{f} = 0.
\end{equation}
Then $f\equiv  0$.
\end{theorem}

\begin{proof}
We first claim that there exists $\varepsilon >0$ such that
\begin{equation}
\label{Econst}
\left(\lambda + \frac{1}{2}\right)H - \varepsilon \left( |A|^2 - \lambda \right) > 0 \quad \text{on}\quad \Sigma.
\end{equation}
By the curvature decay condition \eqref{eq:Decay_A}, there exists a large $r > 0$ such that
\[\abs{A}^2 < \lambda \quad \text{on} \quad \Sigma \backslash B_r(0).\]
Hence, since $H> 0$ and $\lambda> 0$, \eqref{Econst} holds on $\Sigma \backslash B_r(0)$ for all $\varepsilon>0$.

By compactness of $\Sigma \cap B_\rho(0)$, there exists $C > \lambda$ such that
\[\abs{A}^2 < C \quad \text{on} \quad \Sigma.\]
Since $\Sigma$ is mean-convex, by compactness there is $\delta>0$ such that
\begin{align}
\label{Eepscore}
H>\delta \geq \frac{\delta}{C - \lambda} \left(|A|^2 - \lambda\right) \quad \text{on} \quad \Sigma \cap B_r(0).
\end{align}
From this, it follows that \eqref{Econst} holds with $\varepsilon = \frac{\delta\left(\lambda + \frac{1}{2}\right)}{C-\lambda}$.

We are now ready to show that $f \equiv 0$. For simplicity, we denote
\[u := \frac{f}{H+\varepsilon},\]
where $\varepsilon > 0$ is the fixed number such that \eqref{Econst} holds. By \eqref{eq:Quotient}, $u$ satisfies
\begin{equation}
\label{eq:Laplace_u}
\Delta u = -\left \langle \frac{1}{2} F + \frac{2}{H+\varepsilon} \nabla (H+\varepsilon), \nabla u\right\rangle + \frac{u}{H+\varepsilon} \left(\left(\lambda + \frac{1}{2}\right)H - \varepsilon \left( |A|^2 - \lambda\right)\right).
\end{equation}
To show $f \equiv 0$ on $\Sigma$, we argue by contradiction. Suppose there exists $p \in \Sigma$ such that
\[f(p) > 0, \text{ and equivalently } u(p) > 0.\]
Recall that $f$ satisfies \eqref{eq:Decay_f}. As $H + \varepsilon > \varepsilon > 0$, we also have
\[\lim_{r \to\infty}\sup_{\Sigma \backslash B_r(0)} \abs{u} = 0.\]
Hence, there exists a large $r > 0$ such that
\[u(q) < u(p) \quad \text{for any} \quad q \in \Sigma \backslash B_r(0).\]
By compactness of $\Sigma \, \cap \, B_r(0)$, $u$ must achieve an interior maximum at some $p' \in \Sigma \cap B_r(0)$. Note that 
\[u(p') \geq u(p) > 0.\]

Evaluating both sides of \eqref{eq:Laplace_u} at $p'$, we obtain
\begin{align*}
0 & \geq \left(\Delta u\right)(p')\\
& = \underbrace{-\left\langle \frac{1}{2}F + \frac{2}{H+\varepsilon} \nabla (H + \varepsilon),\,\nabla u\right\rangle(p')}_{= 0}\\
& \quad + \frac{u}{H+\varepsilon}(p')\underbrace{\left(\left(\lambda + \frac{1}{2}\right)H - \varepsilon \left( |A|^2 - \lambda\right)\right)(p')}_{\text{positive by \eqref{Econst}}}\\
& > 0,
\end{align*}	
which is clearly a contradiction. Hence, $u \leq 0$ on $\Sigma$. A similar argument (by considering an interior minimum point) shows $u \geq 0$ on $\Sigma$.  We conclude that $u \equiv 0$ (and equivalently $f \equiv 0$) on $\Sigma$.
\end{proof}

\section{Asymptotically Conical Ends}
In this section, we define what it means for an end of a hypersurface $\Sigma^{n}\subset \R^{n+1}$ (which need not be a self-expander) to be asymptotically conical and collect some geometric properties about such ends. 

\begin{definition}
\label{Dcone}
Given $0<\alpha\leq \pi/2$, the \emph{cone of angle alpha}, $\cone \subset \R^{n+1}$ is 
	\[ \cone = \{0\}\cup \left \{ 0\neq x\in \R^{n+1} : \left\langle \frac{x}{ | x|  }, e_{n+1}\right \rangle  = \cos \alpha \right\}.\]
\end{definition}

Clearly $\cone$ is orientable; let $\check\nu$ be the normal field on $\cone$ satisfying $\langle \check\nu, e_{n+1}\rangle> 0$.  From now on, we use the symbol $\;\check{}\;$ to denote geometric quantities of $\cone$.

\begin{definition}[Asymptotically Conical Ends]
\label{Ddecay}
We say $E\subset \Sigma^n$ is \emph{$C^k$-asymptotic to $\cone$} if there exists $\rho> 0$ and  $u\in C^k(\cone\setminus B_\rho(0))$ such that $E$ can be parametrized as a normal graph over $\cone$ of the form \begin{align*}
F(p) &= p + u(p) \, \check{\nu}(p),
\end{align*}
and $u$ satisfies
\begin{align*}
 \left|\check{\nabla}^j u(x) \right|_{\check{g}}  = o\left( |x|^{-j}\right) \text{ for } j = 0, 1, 2, \ldots, k,
\end{align*}
where $\check{\nabla}$ denotes the covariant derivative of $\cone$.
\end{definition}

The goal of this section is to prove the following.
\begin{proposition}
\label{lma:C2Decay}
Suppose $\Sigma$ is a hypersurface with an end $E$ which is $C^2$-asymptotic to $\cone$.  Let $R$ be a rotational Killing vector field about the axis of $\cone$. Then
\begin{enumerate}
\item $\displaystyle{|A|^2 = \frac{(n-1)\cot^2\alpha}{r^2} + o(r^{-2})}.$
\item $\displaystyle{\lim_{r\to \infty} \sup_{E\setminus B_r(0)} |f_R| = 0}.$
\end{enumerate}
\end{proposition}

Before proving Proposition \ref{lma:C2Decay}, we define a convenient parametrization $\check F$ of $\cone\setminus \{ 0\} $.  Let $\Phi(\theta_1, \dots, \theta_{n-1})$ be a parametrization of the unit $S^{n-1}$ (e.g., using hyperspherical coordinates) satisfying
\begin{align}
\label{Elambda}
\left\langle \frac{\partial \Phi}{\partial \theta_i}, \frac{\partial \Phi}{\partial \theta_j}\right \rangle = \lambda_i\delta_{ij}, \quad
\left\langle \frac{\partial^2 \Phi}{\partial \theta_i \partial \theta_j}, \Phi\right \rangle = - \lambda_i \delta_{ij} ,
\end{align}
where $\lambda_i, i=1, \dots, n-1$ is a function of $\omega: = (\theta_1, \ldots, \theta_{n-1})$ satisfying $0< \lambda_i \leq 1$.

Define $\check{F}: \R_+\times S^{n-1} \rightarrow \R^{n+1}$ by 
\begin{align*}
\check{F}(r, \omega) = \left(r\cos\alpha\right) e_{n+1} + \left(r\sin \alpha\right) \Phi(\omega).
\end{align*}

We have 
\begin{align*}
\frac{\partial \check{F}}{\partial r} &= \left(\cos \alpha\right) e_{n+1}+ \left(\sin \alpha\right) \Phi(\omega)\\
\frac{\partial \check{F}}{\partial \theta_i} &= \left(r\sin \alpha\right) \frac{\partial \Phi}{\partial \theta_i}.
\end{align*}

 The first fundamental form of $\cone$ is given by
\begin{align*}
\check{g}_{rr} & = \left\langle \frac{\partial \check{F}}{\partial r}, \frac{\partial \check{F}}{\partial r}\right\rangle = 1 & \check{g}_{rj} & = \left\langle \frac{\partial \check{F}}{\partial r}, \frac{\partial \check{F}}{\partial \theta_j}\right\rangle = 0\\
\check{g}_{ir} & = \left\langle \frac{\partial \check{F}}{\partial\theta_i}, \frac{\partial \check{F}}{\partial r}\right\rangle = 0 & \check{g}_{ij} & = \left\langle \frac{\partial \check{F}}{\partial \theta_i} , \frac{\partial \check{F}}{\partial \theta_j} \right\rangle = \left(r^2\sin^2\alpha\right)\lambda_i\delta_{ij}.
\end{align*}

It is easily verified that
\begin{align}
\label{Enu}
\check{\nu}(r, \omega) &= -\left(\cos \alpha\right)\Phi(\omega) + \left(\sin \alpha\right)e_{n+1}.
\end{align}
Note also that
\[\left\{\check{\nu}, \, \frac{\partial\check{F}}{\partial r}, \, \frac{\partial\Phi}{\partial\theta_i}\right\}_{i=1}^{n-1}\]
is an orthogonal frame for $\R^{n+1}$ at every point on $\cone\setminus \{0\}$; $\check{\nu}$ and $\frac{\partial\check{F}}{\partial r}$ have length $1$, although $\frac{\partial\Phi}{\partial\theta_i}$ may not.

In $(r, \omega)$ coordinates, the second fundamental form of $\cone$ is given by
\begin{align*}
\check{h}_{rr} & = 0 & \check{h}_{rj} & = 0\\
\check{h}_{ir} & = 0 & \check{h}_{ij} & = \left(r\sin\alpha\cos\alpha\right)\lambda_i \delta_{ij}.	
\end{align*}
From these computations, it is clear that the $\cone$ has $n-1$ principal curvatures equal to $\frac{\cot\alpha}{r}$ and one principal curvature equal to $0$. Therefore
\begin{align*}
\abs{\check{A}}^2 & = \frac{(n-1)\cot^2 \alpha}{r^2}.
\end{align*}

\begin{proof}[Proof of Proposition \ref{lma:C2Decay}]

 Straightforward computations show
\begin{align*}
\frac{\partial F}{\partial r} &= \frac{\partial\check{F}}{\partial r} +\frac{\partial u}{\partial r}\check{\nu} \\
\frac{\partial F}{\partial \theta_i} &= \frac{\partial u}{\partial \theta_i} \check{\nu} + r\left(\sin \alpha - \frac{u}{r} \cos \alpha\right) \frac{\partial \Phi}{\partial \theta_i}.
\end{align*}
For abbreviation, we define
\[Z(r,\omega) = \sin \alpha - \frac{u(r,\omega)}{r}\cos\alpha.\]
In $(r, \omega)$ coordinates, the first fundamental form of $E$ is given by
\begin{align*}
g_{rr} & = 1 + \left(\frac{\partial u}{\partial r}\right)^2 & g_{rj} & = \frac{\partial u}{\partial r}\frac{\partial u}{\partial\theta_j}\\
g_{ir} & = \frac{\partial u}{\partial\theta_i}\frac{\partial u}{\partial r} & g_{ij} & = r^2Z^2\lambda_i\delta_{ij} + \frac{\partial u}{\partial\theta_i}\frac{\partial u}{\partial\theta_j}.
\end{align*}

By direct computation (by taking inner products with $\frac{\partial F}{\partial r}$ and $\frac{\partial F}{\partial \theta_i}$),
\begin{align}
\label{EN}
N = Z \check{\nu} - \frac{1}{r}\sum_{i=1}^{n-1}\frac{\partial u}{\partial \theta_i} \frac{\partial \Phi}{\partial \theta_i} - Z\frac{\partial u}{\partial r} \frac{\partial\check{F}}{ \partial r}
\end{align}
is a (not unit) normal field to $E$. 

The $C^2$-asymptotic condition $\abs{\check{\nabla}^j u(x)} = o(\abs{x}^{-j})$ for $j = 0, 1, 2$ is equivalent to
\[\left(\frac{1}{r}\frac{\partial}{\partial\theta_{i_1}}\right)\cdots\left(\frac{1}{r}\frac{\partial}{\partial\theta_{i_l}}\right)\left(\frac{\partial^k}{\partial r^k}\right) u = o(r^{-(k+l)})\]
for any nonnegative $k$, $l$ such that $k + l \leq 2$.  This implies
\begin{equation}
\label{Eudecay}
\begin{aligned}
u & = o(1) & \frac{\partial u}{\partial r} & = o(r^{-1}) & \frac{\partial^2 u}{\partial r^2} & = o(r^{-2})\\
\frac{\partial u}{\partial\theta_i} & = o(1) & \frac{\partial^2 u}{\partial r\partial \theta_i} & = o(r^{-1}) & \frac{\partial^2 u }{\partial\theta_i\partial\theta_j} & = o(1).
\end{aligned}
\end{equation}

Under these asymptotics, it is easy to see that
\[Z = \sin\alpha + o(r^{-1}) \quad \text{ and hence } \quad N = \left(\sin\alpha\right)\check{\nu} + o(r^{-1}).\]

Consider now $f_R : = \langle R(F), \nu \rangle$. Since the cone $\cone$ is rotationally symmetric,
\[\check{f}_R := \langle R(\check{F}), \check{\nu}\rangle \equiv 0.\]
Therefore, the function $f_R$ satisfies
\begin{align*}
\abs{f_R} & = \abs{f_R - \check{f}_R} = \abs{\langle R(F), \nu\rangle - \langle R(\check{F}), \check{\nu}\rangle}\\
& \leq \abs{R(F) - R(\check{F})}\abs{\nu} + \abs{R(\check{F})}\abs{\nu - \check{\nu}}.
\end{align*}
There exists $C$ (depending only on $R$ and $n$) such that
\[\abs{R(F) - R(\check{F})} \leq C\abs{F - \check{F}} = C\abs{u} = o(1).\]
Recalling that 
\[\nu = \frac{N}{\abs{N}} = \frac{\left(\sin\alpha\right)\check{\nu} + o(r^{-1})}{\sin\alpha + o(r^{-1})} = \check{\nu} + o(r^{-1}),\]
we have
\[\abs{R(\check{F})}\abs{\nu - \check{\nu}} \leq C\abs{\check{F}} \abs{\nu - \check{\nu}} = o(1).\]
This concludes the proof of part (1). 

For part (2), we first claim that the coefficients of the second fundamental form of $E$ in $(r, \omega)$ coordinates satisfy
\begin{equation}
\label{Eh}
\begin{aligned}
h_{rr} & = \frac{Z}{\abs{N}}\frac{\partial^2u}{\partial r^2}\\
h_{rj} & = \frac{1}{\abs{N}}\left[Z\frac{\partial^2u}{\partial r \partial\theta_j} + \left(\left(\cos\alpha\right)\frac{\partial u}{\partial r} - \sin\alpha\right)\frac{\lambda_j}{r}\frac{\partial u}{\partial\theta_j}\right]\\
h_{ij} & =  \frac{1}{\abs{N}} \left[\frac{\partial^2u}{\partial\theta_i\partial\theta_j}Z + \frac{(\lambda_i + \lambda_j)\cos\alpha}{r}\frac{\partial u}{\partial\theta_i}\frac{\partial u}{\partial\theta_j} + rZ\left\langle \frac{\partial^2\Phi}{\partial\theta_i\partial\theta_j},\,N \right\rangle\right].
\end{aligned}
\end{equation}
We omit the straightforward calculations of the first two items above, but compute $h_{ij}$ in detail. First note that
\[\frac{\partial\check{\nu}}{\partial\theta_j} = -\left(\cos\alpha\right)\frac{\partial\Phi}{\partial\theta_j},\]
so direct computations establish that
\[\frac{\partial^2 F}{\partial\theta_i\partial\theta_j} = \frac{\partial^2u}{\partial\theta_i\partial\theta_j}\check{\nu} - \left(\cos\alpha\right)\left(\frac{\partial u}{\partial\theta_i}\frac{\partial\Phi}{\partial\theta_j}+\frac{\partial u}{\partial\theta_j}\frac{\partial\Phi}{\partial\theta_i}\right) + rZ\frac{\partial^2\Phi}{\partial\theta_i\partial\theta_j}.\]
This shows
\begin{align*}
h_{ij} & = \left\langle \frac{\partial^2F}{\partial\theta_i\partial\theta_j}, \frac{N}{\abs{N}}\right\rangle\\
& = \frac{1}{\abs{N}} \left[\frac{\partial^2u}{\partial\theta_i\partial\theta_j}Z + \frac{(\lambda_i + \lambda_j)\cos\alpha}{r}\frac{\partial u}{\partial\theta_i}\frac{\partial u}{\partial\theta_j} + rZ\left\langle \frac{\partial^2\Phi}{\partial\theta_i\partial\theta_j},\,N \right\rangle\right], 
\end{align*}
which establishes \eqref{Eh}.

Next, we combine the explicit formulas in \eqref{Eh} with the decay estimates \eqref{Eudecay} to estimate $|A|^2$.  By combining \eqref{Elambda}, \eqref{Enu} and \eqref{EN}, we find that
\begin{align*}
rZ\left \langle \frac{\partial^2 \Phi}{\partial \theta_i\theta_j}, N\right \rangle 
&= (rZ\sin\alpha \cos \alpha)\lambda_i \delta_{ij} + o(1), 
\end{align*}
 hence
\begin{align*}
h_{ij} &= (r\sin \alpha \cos \alpha) \lambda_i \delta_{ij} +o(1).
\end{align*}
Using the asympotics in \eqref{Eudecay} with \eqref{Eh}, the second fundamental form of $E$ satisfies
\begin{align*}
h_{rr} & = o(r^{-2}) & h_{rj} & = o(r^{-1})\\
h_{ir} & = o(r^{-1}) & h_{ij} &= (r \sin \alpha \cos \alpha) \lambda_i \delta_{ij} +o(1).
\end{align*}

In matrix form, the first fundamental form of $\Sigma$ can be written
\[
[g] = \underbrace{\begin{bmatrix} 1 & 0 \\ 0 & r^2Z^2\lambda_i\delta_{ij}\end{bmatrix}}_{=: M} + \underbrace{\begin{bmatrix}\frac{\partial u}{\partial r} \\ \frac{\partial u}{\partial\theta_i}\end{bmatrix}}_{=:\eta}\begin{bmatrix}\frac{\partial u}{\partial r} & \frac{\partial u}{\partial\theta_j}\end{bmatrix}.
\]
By Sherman-Morrison's formula, the inverse of $[g]$ is given by
\[
[g]^{-1} = M^{-1} - \frac{M^{-1}\eta\eta^TM^{-1}}{1+\eta^T M^{-1}\eta},
\]
where $M$ and $\eta$ are defined as above.  By direct computations, we see that
\begin{align*}
\eta^T M^{-1}\eta & = \left(\frac{\partial u}{\partial r}\right)^2 + \sum_{i=1}^{n-1}\frac{1}{r^2Z^2\lambda_i}\left(\frac{\partial u}{\partial\theta_i}\right)^2 = o(r^{-2})\\
M^{-1}\eta\eta^T & = \begin{bmatrix}\left(\frac{\partial u}{\partial r}\right)^2 & \frac{\partial u}{\partial r}\frac{\partial u}{\partial\theta_j} \\ \frac{1}{r^2Z^2\lambda_i}\frac{\partial u}{\partial\theta_i}\frac{\partial u}{\partial r} & \frac{1}{r^2Z^2\lambda_i}\frac{\partial u}{\partial\theta_i}\frac{\partial u}{\partial\theta_j} \end{bmatrix} = \begin{bmatrix} o(r^{-2}) & o(r^{-1}) \\ o(r^{-3}) & o(r^{-2})\end{bmatrix}\\
M^{-1}[h] & = \begin{bmatrix}h_{rr} & h_{rj} \\ \frac{1}{r^2Z^2\lambda_i}h_{ir} & \frac{1}{r^2Z^2\lambda_i}h_{ij}\end{bmatrix} = \begin{bmatrix}o(r^{-2}) & o(r^{-1}) \\ o(r^{-2}) & \frac{\cot \alpha}{r}\delta_{ij} +o (r^{-1})\end{bmatrix}.
\end{align*}
Combining all these asymptotics, we deduce finally that
\begin{align*}
[g]^{-1}[h][g]^{-1}[h] &= \begin{bmatrix} o(r^{-4}) & o(r^{-2}) \\ o(r^{-4}) & \frac{\cot^2 \alpha}{r^2} \delta_{ij} + o(r^{-2})\end{bmatrix},
\end{align*}
 and so 
 \begin{align*}
 \abs{A}^2 := \textup{Trace}\left([g]^{-1}[h][g]^{-1}[h]\right)
&=  \frac{(n-1)\cot^2\alpha}{r^2} + o(r^{-2})\\
&= \abs{\check{A}}^2 + o(r^{-2}).
 \end{align*}
 This concludes the proof of (2).
\end{proof}

\newpage
\section{Main Results}
We now collect the results of the previous sections to prove our main results.  

\begin{proposition}
\label{PAC}
Suppose $\Sigma$ is a complete, mean-convex self-expander. If $\Sigma$ has only one end $E$, which is $C^2$-asymptotic to a cone $\cone$, then $\Sigma$ is rotationally symmetric about the axis of $\cone$.
\end{proposition}
\begin{proof}
Let $R$ be a rotational Killing vector field about the axis of $\cone$, and denote $f_R := \langle R, \nu \rangle$. Since $\Sigma$ has only one end, it follows from Lemma \ref{lma:C2Decay} that 
\[ \lim_{r\to \infty} \sup_{E \backslash B_r(0)} \abs{A} = 0 \quad \text{ and } \quad \lim_{r\to \infty} \sup_{E \backslash B_r(0)} \abs{f_R} = 0.\]
By Lemma \ref{lma:Laplace_f}, the function $f_R$ satisfies
\[\left(L+\frac{1}{2}\right)f_R = \frac{1}{2}f_R,\]
where $L$ is the stability operator defined in \eqref{eq:L}. Therefore, Theorem \ref{Tliouville} implies  $f_R \equiv 0$ on $\Sigma$. Since $R$ is an arbitrary rotational Killing vector field about the axis of $\cone$, we conclude the rotational symmetry.
\end{proof}

Using a similar approach as in Proposition \ref{PAC}, we can establish the same result for mean-convex self-expanders with \emph{two} ends $E_\alpha$ and $E_\beta$, each of which is $C^2$-asymptotic to a round cone. The two asymptotic cones need not have the same cone angle, but they are required to be coaxial so that they have the same set of rotational Killing vector fields. In particular, these include self-expanders asymptotic to a \emph{double} cone. We state the result precisely below, but the proof is omitted since it is essentially the same as in Proposition \ref{PAC}.

\begin{proposition}
\label{PDC}
Suppose $\Sigma$ is a complete, mean-convex self-exapander. If $\Sigma$ has only two ends $E_\alpha$ and $E_\beta$, each end is $C^2$-asymptotic to a round cone, and the cones are coaxial, then $\Sigma$ is rotationally symmetric about the axis of the cones.
\end{proposition}

Finally, we remark that if the self-expander $\Sigma$ is \emph{convex} (which is more restrictive than mean-convex), then one can relax the regularity assumption in Propositions \ref{PAC} and \ref{PDC}. Suppose an end $E$ is just $C^1$-asymptotic to a cone $\cone$, then by the condition that $\abs{\check{\nabla}^j u(x)} = o(\abs{x}^{-j})$ for $j = 0, 1$, we know that
\begin{align*}
u & = o(1) & \frac{\partial u}{\partial r} & = o(r^{-1}) & \frac{\partial u}{\partial\theta_i} & = o(1)	
\end{align*}
as $r \to \infty$. Therefore, we still have
\begin{align*}
Z & = \sin\alpha + o(r^{-1}), & N & = \left(\sin\alpha\right)\check{\nu} + o(r^{-1}) & \nu & = \check{\nu} + o(r^{-1}).
\end{align*}
Using \eqref{eq:MCF_Expander}, we can show the mean curvature decays at the following rate:
\begin{equation}
\label{eq:HDecay}
H = \frac{1}{2}\langle F, \nu\rangle = \frac{1}{2}\left\langle \check{F} + u\check{\nu}, \check{\nu} + o(r^{-1})\right\rangle = o(1).
\end{equation}
Here we have used the fact that $\abs{\check{F}} = r$ and $u = o(1)$.

Now given that $\Sigma$ is convex, \eqref{eq:HDecay} shows each principal curvature of $\Sigma$ converges uniformly to $0$ as $r \to \infty$. In particular, we also have:
\[\lim_{r \to \infty} \sup_{E \backslash B_r(0)} \abs{A} = 0.\]
To guarantee that $f_R \to 0$ uniformly as $r \to \infty$, it is sufficient that the end $E$ be $C^1$-asymptotic to a round cone (see the last part of the proof of Proposition \ref{lma:C2Decay}). Therefore, one may proceed in exactly the same way as in Propositions \ref{PAC} and \ref{PDC} to show the following:

\begin{proposition}
Suppose $\Sigma$ is a complete, convex self-expander. If either:
\begin{itemize}
	\item $\Sigma$ has only one end $E$ which is $C^1$-asymptotic to a round cone $\cone$; or
	\item $\Sigma$ has only two ends $E_\alpha$ and $E_\beta$ which are $C^1$-asymptotic to round cones with the same rotation axis,
\end{itemize}
then $\Sigma$ is rotationally symmetric about the axis of the cone.
\end{proposition}

\bibliographystyle{amsplain}
\bibliography{citations}

\end{document}